\documentclass[12pt]{article}%
\usepackage[intlimits]{amsmath}
\usepackage{amssymb}
\usepackage{latexsym}
\usepackage[T1]{fontenc}
\usepackage[sc]{mathpazo}
\usepackage{color}
\usepackage[colorlinks]{hyperref}
\usepackage{amsfonts}
\usepackage{graphicx}%
\definecolor {refcol}{RGB}{40,0,255}
\hypersetup{colorlinks=true,allcolors=refcol}
\makeatletter
\newfont{\footsc}{cmcsc10 at 8truept}
\newfont{\footbf}{cmbx10 at 8truept}
\newfont{\footrm}{cmr10 at 10truept}
\newtheorem{theorem}{Theorem}

\newtheorem{corollary}[theorem]{Corollary}

\newtheorem{lemma}[theorem]{Lemma}

\newtheorem{problem}[theorem]{Problem}
\newtheorem{proposition}[theorem]{Proposition}
\newtheorem{question}[theorem]{Question}

\newenvironment{proof}[1][Proof]{\noindent{\textbf {#1}  }}  {\hfill$\Box$\bigskip}
\begin{document}

\title{\textbf{Merging the }$A$\textbf{-} \textbf{and} $Q$\textbf{-spectral theories}}
\author{V. Nikiforov\thanks{Department of Mathematical Sciences, University of
Memphis, Memphis TN 38152, USA; \textit{email: vnikifrv@memphis.edu}}}
\maketitle

\begin{abstract}
Let $G$ be a graph with adjacency matrix $A\left(  G\right)  $, and let
$D\left(  G\right)  $ be the diagonal matrix of the degrees of $G.$ The
signless Laplacian $Q\left(  G\right)  $ of $G$ is defined as $Q\left(
G\right)  :=A\left(  G\right)  +D\left(  G\right)  $.

Cvetkovi\'{c} called the study of the adjacency matrix the $A$%
\textit{-spectral theory}, and the study of the signless Laplacian--the
$Q$\textit{-spectral theory}. During the years many similarities and
differences between these two theories have been established. To track the
gradual change of $A\left(  G\right)  $ into $Q\left(  G\right)  $ in this
paper it is suggested to study the convex linear combinations $A_{\alpha
}\left(  G\right)  $ of $A\left(  G\right)  $ and $D\left(  G\right)  $
defined by
\[
A_{\alpha}\left(  G\right)  :=\alpha D\left(  G\right)  +\left(
1-\alpha\right)  A\left(  G\right)  \text{, \ \ }0\leq\alpha\leq1.
\]
This study sheds new light on $A\left(  G\right)  $ and $Q\left(  G\right)  $,
and yields some surprises, in particular, a novel spectral Tur\'{a}n theorem.
A number of challenging open problems are discussed.\medskip



\newpage

\end{abstract}
\tableofcontents

\newpage

\section{Introduction}

Let $G$ be a graph with adjacency matrix $A\left(  G\right)  $, and let
$D\left(  G\right)  $ be the diagonal matrix of the degrees of $G.$ In this
paper we study hybrids of $A\left(  G\right)  $ and $D\left(  G\right)  $
similar to the \emph{signless Laplacian } $Q\left(  G\right)  :=A\left(
G\right)  +D\left(  G\right)  $, put forth by Cvetkovi\'{c} in \cite{Cve05}
and extensively studied since then. For extensive coverage see \cite{CvSi09},
\cite{CvSi10a}, \cite{CvSi10b},\cite{Cve10}, and their references). The
research on $Q\left(  G\right)  $ has shown that it is a remarkable matrix,
unique in many respects. Yet, $Q\left(  G\right)  $ is just the sum of
$A\left(  G\right)  $ and $D\left(  G\right)  $, and the study of $Q\left(
G\right)  $ has uncovered both similarities and differences between $Q\left(
G\right)  $ and $A\left(  G\right)  $. To understand to what extent each of
the summands $A\left(  G\right)  $ and $D\left(  G\right)  $ determines the
properties of $Q\left(  G\right)  $, we propose to study the convex linear
combinations $A_{\alpha}\left(  G\right)  $ of $A\left(  G\right)  $ and
$D\left(  G\right)  $ defined by
\begin{equation}
A_{\alpha}\left(  G\right)  :=\alpha D\left(  G\right)  +\left(
1-\alpha\right)  A\left(  G\right)  ,\text{ \ \ }0\leq\alpha\leq1. \label{def}%
\end{equation}
Many facts suggest that the study of the family $A_{\alpha}\left(  G\right)  $
is long due. To begin with, obviously,
\[
A\left(  G\right)  =A_{0}\left(  G\right)  ,\text{ \ }D\left(  G\right)
=A_{1}\left(  G\right)  ,\text{ \ and \ }Q\left(  G\right)  =2A_{1/2}\left(
G\right)  .
\]
Since $A_{1/2}\left(  G\right)  $ is essentially equivalent to $Q\left(
G\right)  $, in this paper we take $A_{1/2}\left(  G\right)  $ as an exact
substitute for $Q\left(  G\right)  $. With this caveat, one sees that
$A_{\alpha}\left(  G\right)  $ seamlessly joins $A\left(  G\right)  $ to
$D\left(  G\right)  $, with $Q\left(  G\right)  $ being right in the middle of
the range; hence, we can study the gradual changes of $A_{\alpha}\left(
G\right)  ,$ from $A\left(  G\right)  $ to $D\left(  G\right)  $. In this
setup, the matrices $A\left(  G\right)  $, $Q\left(  G\right)  $, and
$D\left(  G\right)  $ can be seen in a new light, and many interesting
problems arise. In particular, we are compelled to investigate the hitherto
uncharted territory $\alpha>1/2$, which holds some surprises, e.g., a novel
version of the spectral Tur\'{a}n theorem (Theorem \ref{tTur} below).

Let us note the crucial identity
\begin{equation}
A_{\alpha}\left(  G\right)  -A_{\beta}\left(  G\right)  =\left(  \alpha
-\beta\right)  L\left(  G\right)  , \label{Maid}%
\end{equation}
where $L\left(  G\right)  $ is the well-studied Laplacian of $G,$ defined as
$L\left(  G\right)  :=D\left(  G\right)  -A\left(  G\right)  $. This neat
relation corroborates the soundness of the definition (\ref{def}).

It is worth pointing out that the family $A_{\alpha}\left(  G\right)  $ is
just a small subset of the \emph{generalized adjacency matrices }defined in
\cite{DaHa03} and the \emph{universal adjacency matrices} defined in
\cite{HaOm11}. However, our restricted definition allows to prove stronger
theorems, which are likely to fail for these more general classes.

The rest of the paper is structured as follows. In the next section we
introduce some notation and recall basic facts about spectra of matrices. In
Section \ref{sbas} we present a few general results about the matrices
$A_{\alpha}\left(  G\right)  .$ Section \ref{slv} deals with the largest
eigenvalue of $A_{\alpha}\left(  G\right)  .$ Section \ref{sex} is dedicated
to spectral extremal problems, which are at the heart of spectral graph
theory. A number of topics are covered in Section \ref{smis}. Finally, in
Section \ref{ssg}, we present the $A_{\alpha}$-spectra of the complete graphs
and the complete bipartite graphs.

\section{Notation and preliminaries}

Let $\left[  n\right]  :=\left\{  1,\ldots,n\right\}  .$ Given a real
symmetric matrix $M,$ write $\lambda_{k}\left(  M\right)  $ for the $k$th
largest eigenvalue of $M.$ For short, we write $\lambda\left(  M\right)  $ and
$\lambda_{\min}\left(  M\right)  $ for the largest and the smallest
eigenvalues of $M.$

Given a graph $G$, we write:

- $V(G)$ and $E(G)$ for the sets of vertices and edges of $G,$ and $v\left(
G\right)  $ for $\left\vert V\left(  G\right)  \right\vert $;

- $\Gamma_{G}\left(  u\right)  $ for the set of neighbors of a vertex $u,$ and
$d_{G}\left(  u\right)  $ for $\left\vert \Gamma_{G}\left(  u\right)
\right\vert $ (the subscript $G$ will be omitted if $G$ is understood);

- $\delta\left(  G\right)  $ and $\Delta\left(  G\right)  $ for the minimum
and maximum degree of $G$;

- $w_{G}\left(  u\right)  $ for the number of walks of length $2$ starting
with the vertex $u,$ i.e., $w_{G}\left(  u\right)  =\sum_{\left\{
u,v\right\}  \in E\left(  G\right)  }d_{G}\left(  v\right)  $;

- $G\left[  X\right]  $ for the subgraph of $G$ induced by a set $X\subset
V\left(  G\right)  $;

- $G-X$ for the graph obtained by deleting the vertices of a set $X\subset
V\left(  G\right)  $.$\medskip$

A \emph{coclique }of $G$ is an edgeless induced subgraph of $G.$ Further,
$K_{n}$ stands for the complete graph of order $n$, and $K_{a,b}$ stands for
the complete bipartite graph with partition sets of sizes $a$ and $b.$ In
particular, $K_{1,n-1}$ denotes the star of order $n$. We write $S_{n,k}$ for
the graph obtained by joining each vertex of a complete graph of order $k$ to
each vertex of an independent set of order $n-k$, that is to say,
$S_{n,k}=K_{k}\vee\overline{K}_{n-k}$.$\medskip$

On many occasions we shall use Weyl's inequalities for eigenvalues of
Hermitian matrices (see, e.g. \cite{HoJo85}, p. 181). Although these
fundamental inequalities have been known for almost a century, it seems that
their equality case was first established by So in \cite{So94}, and his work
was inspired by the paper of Ikebe, Inagaki and Miyamoto \cite{IIM87}$.$

For convenience we state below the complete theorem of Weyl and So:\medskip

\textbf{Theorem WS} \emph{Let }$A$\emph{ and }$B$\emph{ be Hermitian matrices
of order }$n,$\emph{ and let }$1\leq i\leq n$\emph{ and }$1\leq j\leq
n.$\emph{ Then}{%
\begin{align}
\lambda_{i}(A)+\lambda_{j}(B)  &  \leq\lambda_{i+j-n}(A+B),\text{if }i+j\geq
n+1,\label{Wein1}\\
\lambda_{i}(A)+\lambda_{j}(B)  &  \geq\lambda_{i+j-1}(A+B),\text{if }i+j\leq
n+1. \label{Wein2}%
\end{align}
}\emph{In either of these inequalities equality holds if and only if there
exists a nonzero }$n$\emph{-vector that is an eigenvector to each of the three
eigenvalues involved. }\medskip

A simplified version of (\ref{Wein1}) and (\ref{Wein2}) gives%
\begin{equation}
\lambda_{k}\left(  A\right)  +\lambda_{\min}\left(  B\right)  \leq\lambda
_{k}\left(  A+B\right)  \leq\lambda_{k}\left(  A\right)  +\lambda\left(
B\right)  . \label{Wes}%
\end{equation}
\bigskip

We shall need the following simple properties of the Laplacian: \medskip

\textbf{Proposition L }\emph{If }$G$\emph{ is a graph of order }$n\emph{,}%
$\emph{ then }%
\[
\lambda\left(  L\left(  G\right)  \right)  \leq n\text{\textbf{ \ \ }and
\ }\lambda_{\min}\left(  L\left(  G\right)  \right)  =0.
\]
\emph{If }$G$\emph{ is connected, then every eigenvector of }$L\left(
G\right)  $ \emph{to the eigenvalue }$0$ \emph{is constant}$.\medskip$

Recall that a real symmetric matrix $M$ is called \emph{positive semidefinite
}if $\lambda_{\min}\left(  M\right)  \geq0.$ Likewise $M$ is called
\emph{positive definite }if $\lambda_{\min}\left(  M\right)  >0.$

\section{\label{sbas}Basic properties of $A_{\alpha}\left(  G\right)  $}

Given a graph $G$ of order $n,$ it is obvious that the system of
eigenequations for the matrix $A_{\alpha}\left(  G\right)  $ is%
\begin{equation}
\lambda x_{k}=\alpha d_{G}\left(  k\right)  x_{k}+\left(  1-\alpha\right)
\sum_{\left\{  i,k\right\}  \in E\left(  G\right)  }x_{i},\text{ \ \ }1<k\leq
n. \label{eeq}%
\end{equation}

\subsection{The quadratic form $\left\langle A_{\alpha}\mathbf{x}%
,\mathbf{x}\right\rangle $}

If $G$ is a graph of order $n$ with $A_{\alpha}\left(  G\right)  =A_{\alpha},$
and $\mathbf{x}:=\left(  x_{1},\ldots,x_{n}\right)  $ is a real vector, the
quadratic form $\left\langle A_{\alpha}\mathbf{x},\mathbf{x}\right\rangle $
can be represented in several equivalent ways, for example,%
\begin{align}
\left\langle A_{\alpha}\mathbf{x},\mathbf{x}\right\rangle  &  =\sum_{\left\{
u,v\right\}  \in E\left(  G\right)  }(\alpha x_{u}^{2}+2\left(  1-\alpha
\right)  x_{u}x_{v}+\alpha x_{v}^{2}),\label{q1}\\
\left\langle A_{\alpha}\mathbf{x},\mathbf{x}\right\rangle  &  =\left(
2\alpha-1\right)  \sum_{u\in V\left(  G\right)  }x_{u}^{2}d\left(  u\right)
+\left(  1-\alpha\right)  \sum_{\left\{  u,v\right\}  \in E\left(  G\right)
}\left(  x_{u}+x_{v}\right)  ^{2},\label{q2}\\
\left\langle A_{\alpha}\mathbf{x},\mathbf{x}\right\rangle  &  =\alpha
\sum_{u\in V\left(  G\right)  }x_{u}^{2}d\left(  u\right)  +2\left(
1-\alpha\right)  \sum_{\left\{  u,v\right\}  \in E\left(  G\right)  }%
x_{u}x_{v}. \label{q3}%
\end{align}
Each of these representations can be useful in proofs.

Since $A_{\alpha}\left(  G\right)  $ is a real symmetric matrix, Rayleigh's
principle implies that

\begin{proposition}
\label{proRP}If $\alpha\in\left[  0,1\right]  $ and $G$ is a graph of order
$n$ with $A_{\alpha}\left(  G\right)  =A_{\alpha},$ then
\begin{equation}
\lambda\left(  A_{\alpha}\right)  =\max_{\left\Vert \mathbf{x}\right\Vert
_{2}=1}\left\langle A_{\alpha}\mathbf{x},\mathbf{x}\right\rangle \text{
\ \ and \ \ \ }\lambda_{\min}\left(  A_{\alpha}\right)  =\min_{\left\Vert
\mathbf{x}\right\Vert _{2}=1}\left\langle A_{\alpha}\mathbf{x},\mathbf{x}%
\right\rangle . \label{RP}%
\end{equation}
Moreover, if $\mathbf{x}$ is a unit $n$-vector, then $\lambda\left(
A_{\alpha}\right)  =\left\langle A_{\alpha}\mathbf{x},\mathbf{x}\right\rangle
$ if and only if $\mathbf{x}$ is an eigenvector to $\lambda\left(  A_{\alpha
}\right)  ,$ and $\lambda_{\min}\left(  A_{\alpha}\right)  =\left\langle
A_{\alpha}\mathbf{x},\mathbf{x}\right\rangle $\textbf{ }if and only if
$\mathbf{x}$ is an eigenvector to $\lambda_{\min}\left(  A_{\alpha}\right)  $.
\end{proposition}

These relations yield the following familiar relations:

\begin{proposition}
If $\alpha\in\left[  0,1\right)  $ and $G$ is a graph with $A_{\alpha}\left(
G\right)  =A_{\alpha}$, then
\begin{align*}
\lambda\left(  A_{\alpha}\right)   &  =\max\left\{  \lambda\left(  A_{\alpha
}\left(  H\right)  \right)  :H\text{ is a component of }G\right\}  ,\\
\lambda_{\min}\left(  A_{\alpha}\right)   &  =\min\left\{  \lambda_{\min
}\left(  A_{\alpha}\left(  H\right)  \right)  :H\text{ is a component of
}G\right\}  .
\end{align*}

\end{proposition}

\textbf{Caution: }\emph{If }$G$\emph{ is disconnected, }$\lambda\left(
A_{\alpha}\right)  $\emph{ can be attained on different components of }%
$G,$\emph{ depending on }$\alpha.$\emph{ For example, let }$k\geq2$\emph{ be
an integer and let }$G$\emph{ be the disjoint union of }$K_{3k+1,3k+1},$\emph{
}$K_{3,3k^{2}},$\emph{ and }$K_{1,3k^{2}+1}$\emph{. Calculating the largest
eigenvalues of }$A_{0}$\emph{, }$A_{1/2},$\emph{ and }$A_{1}$\emph{ for each
of the three components of }$G,$\emph{ we get the following table:}%
\[%
\begin{tabular}
[c]{|c|c|c|c|c|}\hline
& $K_{3k+1,3k+1}$ & $K_{3,3k^{2}}$ & $K_{1,3k^{2}+1}$ & $G$\\\hline
$\lambda\left(  A_{0}\right)  $ & $3k+1$ & $3k$ & $\sqrt{3k^{2}+1}$ &
$3k+1$\\\hline
$\lambda\left(  A_{1/2}\right)  $ & $3k+1$ & $\left(  3k^{2}+3\right)  /2$ &
$\left(  3k^{2}+1\right)  /2$ & $\left(  3k^{2}+3\right)  /2$\\\hline
$\lambda\left(  A_{1}\right)  $ & $3k+1$ & $3k^{2}$ & $3k^{2}+1$ & $3k^{2}%
+1$\\\hline
\end{tabular}
\ \ \ \ \
\]
\emph{Hence }$\lambda\left(  A_{\alpha}\left(  G\right)  \right)  $\emph{ may
be attained on each of the components of }$G,$\emph{ depending on }%
$\alpha.\medskip$

\subsection{Monotonicity of $\lambda_{k}\left(  A_{\alpha}\left(  G\right)
\right)  $ in $\alpha$}

In this subsection we shall show that $\lambda_{k}\left(  A_{\alpha}\left(
G\right)  \right)  $ is nondecreasing in $\alpha$ for any $k.$ For a start
note that if $G$ is a $d$-regular graph of order $n$, then
\[
A_{\alpha}\left(  G\right)  =\alpha dI_{n}+\left(  1-\alpha\right)  A\left(
G\right)  ,
\]
and so there is a linear correspondence between the spectra of $A_{\alpha
}\left(  G\right)  $ and of $A\left(  G\right)  $%
\begin{equation}
\lambda_{k}\left(  A_{\alpha}\left(  G\right)  \right)  =\alpha d+\left(
1-\alpha\right)  \lambda_{k}\left(  A\left(  G\right)  \right)  ,\text{
\ }1\leq k\leq n. \label{reg}%
\end{equation}
In particular, if $G$ is a $d$-regular graph, then $\lambda\left(  A_{\alpha
}\left(  G\right)  \right)  =d$ for any $\alpha\in\left[  0,1\right]  .$
Moreover, if $G$ is regular and connected graph of order $n$, equations
(\ref{reg}) imply that $\lambda_{k}\left(  A\left(  G\right)  \right)  $ is
increasing in $\alpha$ for any $2\leq k\leq n.$ It turns out that the latter
property is essentially valid for any graph:

\begin{proposition}
\label{pro1}Let $1\geq\alpha>\beta\geq0.$ If $G$ is a graph of order $n$ with
$A_{\alpha}\left(  G\right)  =A_{\alpha}$ and $A_{\beta}\left(  G\right)
=A_{\beta},$ then
\begin{equation}
\lambda_{k}\left(  A_{\alpha}\right)  -\lambda_{k}\left(  A_{\beta}\right)
\geq0 \label{in5}%
\end{equation}
for any $k\in\left[  n\right]  .$ If $G$ is connected, then inequality
(\ref{in5}) is strict, unless $k=1$ and $G$ is regular.
\end{proposition}

\begin{proof}
Identity (\ref{Maid}), inequality (\ref{Wes}), and Proposition L imply that%
\begin{equation}
\lambda_{k}\left(  A_{\alpha}\right)  -\lambda_{k}\left(  A_{\beta}\right)
\geq\left(  \alpha-\beta\right)  \lambda_{\min}\left(  L\left(  G\right)
\right)  =0. \label{in2}%
\end{equation}
If $G$ is connected and equality holds in (\ref{in2}), Theorem WS implies that
$\lambda_{k}\left(  A_{\beta}\right)  $, $\lambda_{k}\left(  A_{\alpha
}\right)  $, and $\lambda_{\min}\left(  L\left(  G\right)  \right)  $ have a
common eigenvector, which by Proposition L must be constant, say the all-ones
vector $\mathbf{j}_{n}.$ Now, Proposition \ref{proPF} implies that $k=1,$ and
the eigenequations (\ref{eeq}) imply that $G$ is regular.
\end{proof}

With the premises of Proposition \ref{pro1}, note also that
\[
\lambda_{k}\left(  A_{\alpha}\right)  -\lambda_{k}\left(  A_{\beta}\right)
\leq\left(  \alpha-\beta\right)  n,
\]
and so we arrive at:

\begin{proposition}
If $G$ is a graph, with $A_{\alpha}\left(  G\right)  =A_{\alpha}$, then the
function $\lambda_{k}\left(  A_{\alpha}\right)  $ is Lipschitz continuous in
$\alpha$ for any $k\in\left[  n\right]  .$ Furthermore, $\lambda\left(
A_{\alpha}\right)  $ is convex in $\alpha,$ and $\lambda_{\min}\left(
A_{\alpha}\right)  $ is concave in $\alpha.$
\end{proposition}

Let us note that the convexity of $\lambda\left(  A_{\alpha}\right)  $ and the
concavity of $\lambda_{\min}\left(  A_{\alpha}\right)  $ follow from
inequalities (\ref{Wes}).

\begin{question}
If $n\geq k\geq1,$ is $f\left(  \alpha\right)  =\lambda_{k}\left(  A_{\alpha
}\right)  $ differentiable in $\alpha$?
\end{question}

\subsection{Positive semidefinitness of $A_{a}$}

An important property of the signless Laplacian $Q\left(  G\right)  $ is that
it is positive semidefinite. This is certainly not true for $A_{\alpha}\left(
G\right)  $ if $\alpha$ is sufficiently small, but if $\alpha\geq1/2,$ then
$A_{\alpha}\left(  G\right)  $ is similar to $Q\left(  G\right)  $:

\begin{proposition}
If $\alpha>1/2$, and $G$ is a graph, then $A_{\alpha}\left(  G\right)  $ is
positive semidefinite. If $G$ has no isolated vertices, then $A_{\alpha
}\left(  G\right)  $ is positive definite.
\end{proposition}

\begin{proof}
Let $\mathbf{x}:=\left(  x_{1},\ldots,x_{n}\right)  $ be a nonzero vector. If
$\alpha>1/2,$ then for any edge $\left\{  u,v\right\}  \in E$, we see that%
\begin{equation}
\left\langle A_{\alpha}\left(  G\right)  \mathbf{x},\mathbf{x}\right\rangle
\geq\left(  1-\alpha\right)  \left(  x_{u}+x_{v}\right)  ^{2}+\left(
2\alpha-1\right)  x_{u}^{2}+\left(  2a-1\right)  x_{v}{}^{2}\geq0. \label{bo}%
\end{equation}
Hence $A_{\alpha}\left(  G\right)  $ is positive semidefinite. Now, suppose
that $G$ has no isolated vertices. Select a vertex $u$ with $x_{u}\neq0$ and
let $\left\{  u,v\right\}  \in E.$ Then we have strict inequality in
(\ref{bo}) and so $A_{\alpha}\left(  G\right)  $ is positive definite.
\end{proof}

Obviously Proposition \ref{pro1} implies that if $A_{\alpha}\left(  G\right)
$ is positive (semi)definite for some $\alpha,$ then $A_{\beta}\left(
G\right)  $ is positive (semi)definite for any $\beta>a$. This observation
leads to the following problem:

\begin{problem}
Given a graph $G,$ find the smallest $\alpha$ for which $A_{\alpha}\left(
G\right)  $ is positive semidefinite.
\end{problem}

For example, if $G$ is the complete graph $K_{n},$ we have $\lambda_{\min
}\left(  A_{\alpha}\left(  K_{n}\right)  \right)  =n\alpha-1,$ and so
$A_{\alpha}\left(  K_{n}\right)  $ is positive semidefinite if and only if
$\alpha\geq1/n.$ This example can be generalized as follows:

\begin{proposition}
Let $G$ be a regular graph with chromatic number $r$. If $\alpha<1/r,$ then
$A_{\alpha}\left(  G\right)  $ is not positive semidefinite$.$
\end{proposition}

\begin{proof}
Let $G$ be a $d$-regular graph and let $A$ be its adjacency matrix. Hoffman's
bound \cite{Hof70} \ implies that
\[
\lambda_{\min}\left(  A\right)  \leq-\frac{\lambda\left(  A\right)  }%
{r-1}=-\frac{d}{r-1}.
\]
Hence, (\ref{reg}) implies that%
\[
\lambda_{\min}\left(  A_{\alpha}\left(  G\right)  \right)  \leq\alpha
d-\left(  1-\alpha\right)  \frac{d}{r-1}=\left(  \alpha-\frac{1}{r}\right)
\frac{rd}{r-1}<0,
\]
completing the proof.
\end{proof}

\subsection{Some degree based bounds}

It is not an exaggeration to say that degree bounds are the most used bounds
in spectral graph theory. We give a few such bounds for $A_{\alpha}\left(
G\right)  $, the first of which follows from Proposition \ref{pro1}.

\begin{proposition}
Let $G$ is a graph of order $n$ with degrees $d\left(  1\right)  \geq
\cdots\geq d\left(  n\right)  $ and with $A_{\alpha}\left(  G\right)
=A_{\alpha}$. If $k\in\left[  n\right]  ,$ then
\[
\lambda_{k}\left(  A_{\alpha}\right)  \leq d\left(  k\right)  .
\]
In particular, $\lambda\left(  A_{\alpha}\right)  \leq\Delta\left(  G\right)
.$
\end{proposition}

Using an idea of Das \cite{Das10}, the bound $\lambda_{\min}\left(  A_{\alpha
}\right)  \leq\delta\left(  G\right)  $ can be improved further: let $u$ be a
vertex with minimum degree and define the $n$-vector $\mathbf{x}:=\left(
x_{1},\ldots,x_{n}\right)  $ by letting $x_{u}:=1$ and zeroing the other
entries. Then Proposition \ref{proRP} and equation (\ref{q2}) imply that
\[
\lambda_{\min}\left(  A_{\alpha}\right)  \leq\left\langle A_{\alpha}%
\mathbf{x},\mathbf{x}\right\rangle =\left(  2\alpha-1\right)  \delta+\left(
1-\alpha\right)  \delta=\alpha\delta.
\]
But for $\alpha\in\left[  0,1\right)  $ the vector $\mathbf{x}$ does not
satisfy the eigenequations for $\lambda_{\min}\left(  A_{\alpha}\right)  $, so
in this case
\[
\lambda_{\min}\left(  A_{\alpha}\right)  <\alpha\delta.
\]

Further, Weyl's inequality (\ref{Wes}) immediately implies the following bounds:

\begin{proposition}
If $\alpha\in\left[  0,1\right]  $ and $G$ is a graph with $A\left(  G\right)
=A$ and $A_{\alpha}\left(  G\right)  =A_{\alpha},$ then \ \
\[
\alpha\delta+\left(  1-\alpha\right)  \lambda_{k}\left(  A\right)  \leq
\lambda_{k}\left(  A_{\alpha}\right)  \leq\alpha\Delta+\left(  1-\alpha
\right)  \lambda_{k}\left(  A\right)
\]

\end{proposition}

For $\lambda\left(  A_{\alpha}\right)  $ we give a tight lower bound, which
generalizes a result of Lov\'{a}sz (\cite{Lov79}, Problem 11.14):

\begin{proposition}
If $G$ is a graph with $\Delta\left(  G\right)  =\Delta,$ then
\[
\lambda\left(  A_{\alpha}\right)  \geq\frac{1}{2}\left(  \alpha\left(
\Delta+1\right)  +\sqrt{\alpha^{2}\left(  \Delta+1\right)  ^{2}+4\Delta\left(
1-2\alpha\right)  }\right)
\]
If $G$ is connected, equality holds if and only if $G=K_{1,\Delta}.$
\end{proposition}

\begin{proof}
Proposition \ref{prost} gives the spectral radius of $A_{\alpha}$ of a star.
This result, combined with Proposition \ref{proPF}, yields%
\[
\lambda\left(  A_{\alpha}\left(  G\right)  \right)  \geq\lambda\left(
A_{\alpha}\left(  K_{1,\Delta}\right)  \right)  =\frac{1}{2}\left(
\alpha\left(  \Delta+1\right)  +\sqrt{\alpha^{2}\left(  \Delta+1\right)
^{2}+4\Delta\left(  1-2\alpha\right)  }\right)  .
\]
The case of equality also follows from Proposition \ref{proPF}.
\end{proof}

Some algebra can be used to prove a simpler lower bound:

\begin{corollary}
\label{corlo}Let $G$ be a graph with $\Delta\left(  G\right)  =\Delta.$ If
$\alpha\in\left[  0,1/2\right]  ,$ then%
\[
\lambda\left(  A_{\alpha}\left(  G\right)  \right)  \geq\alpha\left(
\Delta+1\right)  .
\]
If $\alpha\in\left[  1/2,1\right)  ,$ then%
\[
\lambda\left(  A_{\alpha}\left(  G\right)  \right)  \geq\alpha\Delta
+1-\alpha.
\]

\end{corollary}

\section{\label{slv}The largest eigenvalue $\lambda\left(  A_{\alpha}\left(
G\right)  \right)  $}

As for the adjacency matrix and the signless Laplacian, the spectral radius
$\lambda\left(  A_{\alpha}\left(  G\right)  \right)  $ of $A_{\alpha}\left(
G\right)  $ is its most important eigenvalue, due to the fact that $A_{\alpha
}\left(  G\right)  $ is nonnegative and so $\lambda\left(  A_{\alpha}\left(
G\right)  \right)  $ has maximal modulus among all eigenvalues of $A_{\alpha
}\left(  G\right)  .$

\subsection{Perron-Frobenius properties of $A_{\alpha}\left(  G\right)  $}

In this subsection we spell out the properties of $\lambda\left(  A_{\alpha
}\left(  G\right)  \right)  $, which follow from the Perron-Frobenius theory
of nonnegative matrices. Observe that if $0\leq\alpha<1$ and $G\ $is a graph,
then $G$ is connected if an only if $A_{\alpha}\left(  G\right)  $ is
irreducible, because irreducibility is not affected by the diagonal entries of
$A_{\alpha}\left(  G\right)  $. Hence, the Perron-Frobenius theory of
nonnegative matrices implies the following properties of $A_{\alpha}\left(
G\right)  $:

\begin{proposition}
\label{proPF}Let $\alpha\in\left[  0,1\right)  ,$ let $G$ be a graph, and let
$\mathbf{x}$ be a nonnegative eigenvector to $\lambda\left(  A_{\alpha}\left(
G\right)  \right)  $:

(a) If $G\ $is connected, then $\mathbf{x}$ is positive and is unique up to scaling.

(b) If $G\ $is not connected and $P$ is the set of vertices with positive
entries in $\mathbf{x},$ then the subgraph induced by $P$ is a union of
components $H$ of $G$ with $\lambda\left(  A_{\alpha}\left(  H\right)
\right)  =\lambda\left(  A_{\alpha}\left(  G\right)  \right)  $.

(c) If $G\ $is connected and $\mu$ is an eigenvalue of $A_{\alpha}\left(
G\right)  $ with a nonnegative eigenvector, then $\mu=$ $\lambda\left(
A_{\alpha}\left(  G\right)  \right)  .$

(d) If $G\ $is connected, and $H$ is a proper subgraph of $G,$ then
$\lambda\left(  A_{\alpha}\left(  H\right)  \right)  <$ $\lambda\left(
A_{\alpha}\left(  G\right)  \right)  $ for any $\alpha\in\left[  0,1\right)
.$
\end{proposition}

A useful corollary can be deduced for the join of two regular graphs:

\begin{proposition}
\label{proj}Let $G_{1}$ be a $r_{1}$-regular graph of order $n_{1},$ and
$G_{2}$ be a $r_{2}$-regular graph of order $n_{2}.$ Then%
\[
\lambda\left(  A_{\alpha}\left(  G_{1}\vee G_{2}\right)  \right)
=\lambda\left(
\begin{array}
[c]{cc}%
r_{1}+\alpha n_{2} & \left(  1-\alpha\right)  ^{2}n_{1}n_{2}\\
1 & r_{2}+\alpha n_{1}%
\end{array}
\right)
\]

\end{proposition}

In turn, Proposition \ref{proj} can be extended to equitable partitions.

Another practical consequence of Proposition \ref{proPF} reads as:

\begin{proposition}
Let $\alpha\in\left[  0,1\right)  $ and let $G$ be a graph with $A_{\alpha
}\left(  G\right)  =A_{\alpha}.$ Let $u,v,w\in V\left(  G\right)  $ and
suppose that $\left\{  u,v\right\}  \in E\left(  G\right)  $ and $\left\{
u,w\right\}  \notin E\left(  G\right)  .$ Let $H$ be the graph obtained from
$G$ by deleting the edge $\left\{  u,v\right\}  $ and adding the edge
$\frac{{}}{{}}\left\{  u,w\right\}  .$ If $\mathbf{x}:=\left(  x_{1}%
,\ldots,x_{n}\right)  $ is a unit eigenvector to $\lambda\left(  A_{\alpha
}\right)  $ such that $x_{u}>0$ and
\[
\left\langle A_{\alpha}\left(  H\right)  \mathbf{x},\mathbf{x}\right\rangle
\geq\left\langle A_{\alpha}\mathbf{x},\mathbf{x}\right\rangle ,
\]
then $\lambda\left(  A_{\alpha}\left(  H\right)  \right)  >\lambda\left(
A_{\alpha}\right)  .$
\end{proposition}

\begin{proof}
Proposition \ref{proRP} implies that immediately that $\lambda\left(
A_{\alpha}\left(  H\right)  \right)  \geq\lambda\left(  A_{\alpha}\right)  ,$
so our goal is to show that equality cannot hold. Assume for a contradiction
that $\lambda\left(  A_{\alpha}\left(  H\right)  \right)  =\lambda\left(
A_{\alpha}\right)  $ and set $\lambda=\lambda\left(  A_{\alpha}\right)  .$
Proposition \ref{proRP} implies that $\mathbf{x}$ is an eigenvector to $H$ and
therefore
\begin{align*}
\lambda x_{w}  &  =\alpha d_{H}\left(  w\right)  x_{w}+\left(  1-\alpha
\right)  \sum_{\left\{  i,w\right\}  \in E\left(  H\right)  }x_{i}\\
&  =\alpha\left(  d_{G}\left(  w\right)  +1\right)  x_{w}+\left(
1-\alpha\right)  x_{u}+\sum_{\left\{  i,w\right\}  \in E\left(  G\right)
}x_{i}\\
&  >\alpha d_{G}\left(  w\right)  x_{w}+\sum_{\left\{  i,w\right\}  \in
E\left(  G\right)  }x_{i},
\end{align*}
contradicting the fact that $\mathbf{x}$ is an eigenvector to $\lambda\left(
A_{\alpha}\right)  $ in $G.$
\end{proof}

\subsection{Eigenvectors to $\lambda\left(  A_{\alpha}\left(  G\right)
\right)  $ and automorphisms}

Knowing the symmetries of a graph $G$ can be quite useful to find the spectral
radius of $\lambda\left(  A_{\alpha}\left(  G\right)  \right)  $. Thus, we say
that $u$ and $v$ are \emph{equivalent in }$G,$ if there exists an automorphism
$p:G\rightarrow G$ such that $p\left(  u\right)  =v.$ Vertex equivalence
implies very useful properties of eigenvectors to $\lambda\left(  A_{\alpha
}\left(  G\right)  \right)  $:

\begin{proposition}
\label{eqth}Let $G\ $be a connected graph of order $n,$ and let $u$ and $v$ be
equivalent vertices in $G$. If $\left(  x_{1},\ldots,x_{n}\right)  $ is an
eigenvector to $\lambda\left(  A_{\alpha}\left(  G\right)  \right)  $, then
$x_{u}=x_{v.}$.
\end{proposition}

\begin{proof}
Let $G\ $be a connected graph with $A_{\alpha}\left(  G\right)  =A_{\alpha}$;
let $\lambda:=\lambda\left(  A_{\alpha}\right)  $ and $\mathbf{x}:=\left(
x_{1},\ldots,x_{n}\right)  $ be a unit nonnegative eigenvector to $\lambda.$
Let $p:G\rightarrow G$ be an automorphism of $G$ such that $p\left(  u\right)
=v.$ Note that $p$ is a permutation of $V\left(  G\right)  $; let $P$ be the
permutation matrix corresponding to $p.$ Since is an automorphism, we have
$P^{-1}A_{\alpha}P=A_{\alpha}$; hence,
\[
P^{-1}A_{\alpha}P\mathbf{x}=\lambda\mathbf{x},
\]
and so $P\mathbf{x}$ is an eigenvector to $A_{\alpha}.$ Since $A_{a}$ is
irreducible, $\mathbf{x}$ is unique, implying that $P\mathbf{x=x}$, and so
$x_{u}=x_{v}$.
\end{proof}

Note that eigenvector entries corresponding to equivalent vertices need not be
equal for disconnected graphs; for example, this not the case if $G$ is a
union of two disjoint copies of an $r$-regular graph. However, Proposition
\ref{eqth} implies the following practical statement:

\begin{corollary}
\label{corEX}If $G\ $is a connected graph and $V\left(  G\right)  $ is
partitioned into equivalence classes by the relation \textquotedblleft$u$ is
equivalent to $v$\textquotedblright, then every eigenvector to $\lambda\left(
A_{\alpha}\right)  $ is constant within each equivalence class.
\end{corollary}

\subsection{A few general bounds on $\lambda\left(  A_{\alpha}\left(
G\right)  \right)  $}

In this section we give a few additional bounds on $\lambda\left(  A_{\alpha
}\right)  .$

\begin{proposition}
\label{pro2}Let $G$ be a graph, with $\Delta\left(  G\right)  =\Delta$,
$A\left(  G\right)  =A$, $D\left(  G\right)  =D$, and $A_{\alpha}\left(
G\right)  =A_{\alpha}$. The following inequalities hold for $\lambda\left(
A_{\alpha}\left(  G\right)  \right)  $:
\begin{align}
\lambda\left(  A_{\alpha}\right)   &  \geq\lambda\left(  A\right)
,\label{bolo}\\
\lambda\left(  A_{\alpha}\right)   &  \leq\alpha\Delta+\left(  1-\alpha
\right)  \lambda\left(  A\right)  . \label{boup}%
\end{align}
If equality holds in (\ref{bolo}), then $G$ has a\ $\lambda\left(  A\right)
$-regular component. Equality in (\ref{boup}) holds if an only if $G$ has a
$\Delta$-regular component.
\end{proposition}

\begin{proof}
Note that inequality (\ref{bolo}) follows from Proposition \ref{pro1}, but we
shall give another proof to deduce the case of equality. Let $H$ be a
component of $G$ such that $\lambda\left(  A\right)  =\lambda\left(  A\left(
H\right)  \right)  .$ Write $h$ for the order of $H$, and let $\left(
x_{1},\ldots,x_{h}\right)  $ be a positive unit vector to $\lambda\left(
A\left(  H\right)  \right)  $. For every edge $\left\{  u,v\right\}  $ of $H,$
the AM-GM inequality implies that
\begin{equation}
2x_{u}x_{v}=2\alpha x_{u}x_{v}+2\left(  1-\alpha\right)  x_{u}x_{v}\leq\alpha
x_{u}^{2}+2\left(  1-\alpha\right)  x_{u}x_{v}+\alpha x_{v}^{2}. \label{in1}%
\end{equation}
Summing this inequality over all edges $\left\{  u,v\right\}  \in E\left(
H\right)  $, and using (\ref{q1}), we get
\[
\lambda\left(  A\right)  =\lambda\left(  A\left(  H\right)  \right)
=\left\langle A\left(  H\right)  \mathbf{x},\mathbf{x}\right\rangle
\leq\left\langle A_{\alpha}\left(  H\right)  \mathbf{x},\mathbf{x}%
\right\rangle \leq\lambda\left(  A_{\alpha}\right)  ,
\]
so (\ref{bolo}) is proved. If equality holds in (\ref{bolo}), then
$x_{1}=\cdots=x_{h},$ hence $H$ is $\lambda\left(  A\right)  $-regular.

Inequality (\ref{boup}) follows by Weyl's inequalities (\ref{Wes}) because
\[
\lambda\left(  A_{\alpha}\right)  \leq\lambda\left(  \alpha D\right)
+\lambda\left(  \left(  1-\alpha\right)  \left(  A\right)  \right)  =\left(
1-\alpha\right)  \lambda\left(  A\right)  +\alpha\Delta,
\]
but we shall give a direct proof based on (\ref{q3}), since it is more
appropriate for the case of equality. Let $H$ be a component of $G$ such that
$\lambda\left(  A_{\alpha}\right)  =\lambda\left(  A_{\alpha}\left(  H\right)
\right)  $ and let $h$ be the order of $H.$ Let $\mathbf{x}:=\left(
x_{1},\ldots,x_{n}\right)  $ be a positive unit eigenvector to $\lambda\left(
A_{\alpha}\left(  H\right)  \right)  .$ We have
\begin{align*}
\lambda\left(  A_{\alpha}\right)   &  =\alpha\sum_{u\in V\left(  H\right)
}x_{u}^{2}d_{G}\left(  u\right)  +2\left(  1-\alpha\right)  \sum_{\left\{
u,v\right\}  \in E\left(  H\right)  }x_{u}x_{v}\\
&  \leq\alpha\Delta\left(  H\right)  \sum_{u\in V\left(  H\right)  }x_{u}%
^{2}+\left(  1-\alpha\right)  \lambda\left(  A\left(  H\right)  \right) \\
&  \leq\alpha\Delta+\left(  1-\alpha\right)  \lambda\left(  A\right)  ,
\end{align*}
proving (\ref{boup}). If equality holds in (\ref{boup}), then $H$ is $\Delta$-regular.

It is not hard to see that if $G$ has a $\Delta$-regular component, then
$\lambda\left(  A\right)  =\Delta=\lambda\left(  A_{\alpha}\right)  $, and so
equality holds in (\ref{boup}).
\end{proof}

Having inequality (\ref{bolo}) in hand, every lower bound of $\lambda\left(
A\right)  $ gives a lower bound on $\lambda\left(  A_{\alpha}\right)  ,$
which, however, is never better than (\ref{bolo}). We mention just two such bounds.

\begin{corollary}
Let $G$ be a graph with $A_{\alpha}\left(  G\right)  =A_{\alpha}$. If $G$ is
of order $n$ and has $m$ edges, then
\[
\lambda\left(  A_{\alpha}\right)  \geq\sqrt{\frac{1}{n}\sum_{u\in V\left(
G\right)  }d_{G}^{2}\left(  u\right)  }\text{ and }\lambda\left(  A_{\alpha
}\right)  \geq\frac{2m}{n}.
\]
Equality holds in the second inequality if and only if $G$ is regular. If
$\alpha>0,$ equality holds in the first inequality if and only $G$ is regular.
\end{corollary}

\begin{proof}
The only difficulty is to prove that if $\alpha>0,$ then the equality
\begin{equation}
\lambda\left(  A_{\alpha}\right)  =\sqrt{\frac{1}{n}\sum_{u\in V\left(
G\right)  }d_{G}^{2}\left(  u\right)  } \label{eq1}%
\end{equation}
implies that $G$ is regular. Indeed, suppose that (\ref{eq1}) holds, which
implies also that
\[
\lambda\left(  A\left(  G\right)  \right)  =\sqrt{\frac{1}{n}\sum_{u\in
V\left(  G\right)  }d_{G}^{2}\left(  u\right)  }.
\]
Let $G_{1},\ldots,G_{k}$ be the components of $G$ and $n_{1},\ldots,n_{k}$ be
their orders. We see that
\begin{align*}
\sum_{u\in V\left(  G\right)  }d_{G}^{2}\left(  u\right)   &  =\lambda
^{2}\left(  A\left(  G\right)  \right)  n\geq\lambda^{2}\left(  A\left(
G_{1}\right)  \right)  n_{1}+\cdots+\lambda^{2}\left(  A\left(  G_{1}\right)
\right)  n_{k}\\
&  \geq\sum_{u\in V\left(  G_{1}\right)  }d_{G_{1}}^{2}\left(  u\right)
+\cdots+\sum_{u\in V\left(  G_{k}\right)  }d_{G_{k}}^{2}\left(  u\right)
=\sum_{u\in V\left(  G\right)  }d_{G}^{2}\left(  u\right)  \text{.}%
\end{align*}
Hence,%
\[
\lambda\left(  A\left(  G_{1}\right)  \right)  =\cdots=\lambda\left(  A\left(
G_{k}\right)  \right)  =\lambda\left(  A\left(  G\right)  \right)  ,
\]
and likewise,%
\[
\lambda\left(  A_{\alpha}\left(  G_{1}\right)  \right)  =\cdots=\lambda\left(
A_{\alpha}\left(  G_{k}\right)  \right)  =\lambda\left(  A_{\alpha}\left(
G\right)  \right)  .
\]
Now, Proposition \ref{pro1} implies that all components of $G$ are regular,
completing the proof.
\end{proof}

A very useful bound in extremal problems about $\lambda\left(  Q\right)  $\ is
the following one
\begin{equation}
\lambda\left(  Q\right)  \leq\max_{v\in V}\left\{  d\left(  u\right)
+\frac{1}{d\left(  u\right)  }\sum_{\left\{  u,v\right\}  \in E\left(
G\right)  }d\left(  v\right)  \right\}  , \label{boMer}%
\end{equation}
with equality if and only if $G$ is regular or semiregular. Bound
(\ref{boMer}) goes back to Merris \cite{Mer98}, whereas the case of equality
has been established by Feng and Yu in \cite{FeYu09}. It is not hard to modify
(\ref{boMer}) for the matrices $A_{\alpha}\left(  G\right)  $:

\begin{proposition}
If $G$ is a graph with no isolated vertices, then%
\begin{equation}
\lambda\left(  A_{\alpha}\left(  G\right)  \right)  \leq\max_{v\in V\left(
G\right)  }\left\{  \alpha d\left(  u\right)  +\frac{1-\alpha}{d\left(
u\right)  }\sum_{\left\{  u,v\right\}  \in E\left(  G\right)  }d\left(
v\right)  \right\}  \label{ubo}%
\end{equation}
and
\begin{equation}
\lambda\left(  A_{\alpha}\left(  G\right)  \right)  \geq\min_{v\in V\left(
G\right)  }\left\{  \alpha d\left(  u\right)  +\frac{1-\alpha}{d\left(
u\right)  }\sum_{\left\{  u,v\right\}  \in E\left(  G\right)  }d\left(
v\right)  \right\}  . \label{lbo}%
\end{equation}
If $\alpha\in\left(  1/2,1\right)  $ and $G$ is connected, equality in
(\ref{ubo}) and (\ref{lbo}) holds if and only if $G$ is regular.
\end{proposition}

\begin{proof}
Let $A_{\alpha}\left(  G\right)  =A_{\alpha}$. Our proof of (\ref{ubo}) and
(\ref{lbo}) uses the idea of Merris. The matrix $D^{-1}A_{\alpha}D$ is similar
to $A_{\alpha}$ and so $\lambda\left(  A_{\alpha}\right)  =\lambda\left(
D^{-1}A_{\alpha}D\right)  $. Since $D^{-1}A_{\alpha}D$ is nonnegative,
$\lambda\left(  D^{-1}A_{\alpha}D\right)  $ is between the smallest and the
largest rowsums of $D^{-1}A_{\alpha}D,$ implying both (\ref{ubo}) and
(\ref{lbo}).

If $G$ is connected, then $A_{\alpha}$ is irreducible and so is $D^{-1}%
A_{\alpha}D.$ Hence, if equality holds in either (\ref{ubo}) and (\ref{lbo}),
then all rowsums of $D^{-1}A_{\alpha}D$ are equal. The remaining part of the
proof uses an idea borrowed from \cite{FeYu09}. For any vertex $v\in V\left(
G\right)  $, set
\[
m\left(  u\right)  =\frac{1}{d\left(  u\right)  }\sum_{\left\{  u,v\right\}
\in E\left(  G\right)  }d\left(  v\right)  .
\]
Fix a vertex \ $u$ and let $v$ be any neighbor of $u$. Now, from
\[
\alpha d\left(  u\right)  +\left(  1-\alpha\right)  m\left(  u\right)  =\alpha
d\left(  v\right)  +\left(  1-\alpha\right)  m\left(  v\right)
\]
we see that
\[
\sum_{\left\{  u,v\right\}  \in E\left(  G\right)  }\alpha d\left(  u\right)
+\left(  1-\alpha\right)  m\left(  u\right)  =\sum_{\left\{  u,v\right\}  \in
E\left(  G\right)  }\alpha d\left(  v\right)  +\left(  1-\alpha\right)
m\left(  v\right)  .
\]
Hence%
\[
\alpha d^{2}\left(  u\right)  +\left(  1-\alpha\right)  d\left(  u\right)
m\left(  u\right)  =\alpha d\left(  u\right)  m\left(  u\right)  +\left(
1-\alpha\right)  \sum_{\left\{  u,v\right\}  \in E\left(  G\right)  }m\left(
v\right)  .
\]
Taking $u$ to be a vertex with maximum degree, we see that
\[
\alpha d^{2}\left(  u\right)  +\left(  1-2\alpha\right)  d\left(  u\right)
m\left(  u\right)  =\left(  1-\alpha\right)  \sum_{\left\{  u,v\right\}  \in
E\left(  G\right)  }m\left(  v\right)  \leq\left(  1-\alpha\right)
d^{2}\left(  u\right)  .
\]
Hence $m\left(  u\right)  \geq d\left(  u\right)  ,$ which is possible only if
all neighbors of $u$ have maximal degree as well. Since $G$ is connected, it
turns out that $G$ is regular.
\end{proof}

\begin{corollary}
For any graph $G,$
\begin{equation}
\lambda\left(  A_{\alpha}\right)  \leq\max_{\left\{  u,v\right\}  \in E\left(
G\right)  }\alpha d\left(  u\right)  +\left(  1-\alpha\right)  d\left(
v\right)  . \label{in3}%
\end{equation}
and
\begin{equation}
\lambda\left(  A_{\alpha}\right)  \geq\min_{\left\{  u,v\right\}  \in E\left(
G\right)  }\alpha d\left(  u\right)  +\left(  1-\alpha\right)  d\left(
v\right)  . \label{in4}%
\end{equation}

\end{corollary}

\textbf{Caution} \emph{If the right side of (\ref{in3} is equal to }$M,$\emph{
and is maximized for }$\left\{  u,v\right\}  \in E\left(  G\right)  $\emph{,
then}
\[
M=\max\left\{  \alpha d\left(  u\right)  +\left(  1-\alpha\right)  d\left(
v\right)  ,\alpha d\left(  v\right)  +\left(  1-\alpha\right)  d\left(
u\right)  \right\}  .
\]

\emph{A similar remark is valid for (\ref{in4}) with appropriate
changes.\medskip}

It seems that equality in (\ref{ubo}) and (\ref{lbo}) holds only if $G$ is
regular, except in the cases $\alpha=0$ and $\alpha=1/2.$ If true, this fact
would need new proof techniques, so we raise the following problem.

\begin{problem}
Find all cases of equality in (\ref{ubo}), (\ref{lbo}), (\ref{in3}), and
(\ref{in4}).
\end{problem}

The last bounds in this sections are in the spirit of (\ref{ubo}) and
(\ref{lbo}):

\begin{proposition}
Let $\alpha\in\left[  0,1\right]  $. If $G$ be a graph of order $n,$ then
\[
\lambda^{2}\left(  A_{\alpha}\left(  G\right)  \right)  \leq\max_{k\in
V\left(  G\right)  }\alpha d_{G}^{2}\left(  k\right)  +\left(  1-\alpha
\right)  w_{G}\left(  k\right)  \text{ }%
\]
and%
\[
\lambda^{2}\left(  A_{\alpha}\left(  G\right)  \right)  \geq\min_{k\in
V\left(  G\right)  }\alpha d_{G}^{2}\left(  k\right)  +\left(  1-\alpha
\right)  w_{G}\left(  k\right)  \text{.}%
\]

\end{proposition}

\begin{proof}
Let $A_{\alpha}:=A_{\alpha}\left(  G\right)  $, $A:=A\left(  G\right)  $,
$D:=D\left(  G\right)  $. First, we show that for any $k\in\left[  n\right]
,$ the $k$th rowsum of $A_{\alpha}^{2}\left(  G\right)  $ is equal to
\[
\alpha d_{G}^{2}\left(  k\right)  +\left(  1-\alpha\right)  w_{G}\left(
k\right)  .
\]
Indeed, for the square of $A_{\alpha},$ we see that
\[
A_{\alpha}^{2}=\alpha^{2}D^{2}+\left(  1-\alpha\right)  ^{2}A^{2}%
+\alpha\left(  1-\alpha\right)  DA+\alpha\left(  1-\alpha\right)  AD.
\]
So for the $k$th rowsum $r_{k}\left(  A_{\alpha}^{2}\right)  $ we find that
\begin{align*}
r_{k}(A_{\alpha}^{2})  &  =\alpha^{2}r_{k}(D^{2})+\left(  1-\alpha\right)
^{2}r_{k}(A^{2})+\alpha\left(  1-\alpha\right)  r_{k}\left(  DA\right)
+\alpha\left(  1-\alpha\right)  r_{k}\left(  AD\right) \\
&  =\alpha^{2}d_{G}^{2}\left(  k\right)  +\left(  1-\alpha\right)  ^{2}%
w_{G}\left(  k\right)  +\alpha\left(  1-\alpha\right)  d_{G}^{2}\left(
k\right)  +\alpha\left(  1-\alpha\right)  w_{G}\left(  k\right) \\
&  =\alpha d_{G}^{2}\left(  k\right)  +\left(  1-\alpha\right)  w_{G}\left(
k\right)  .
\end{align*}
Since \ $\lambda^{2}\left(  A_{\alpha}\right)  =\lambda\left(  A_{\alpha}%
^{2}\right)  ,$ the assertions follow, because $\lambda\left(  A_{\alpha}%
^{2}\right)  $ is between the smallest and the largest rowsums of $A_{\alpha
}^{2}.$
\end{proof}

\section{\label{sex}Some spectral extremal problems}

Recall that the central problem of the classical extremal graph theory is of
the following type: \medskip

\textbf{Problem A }\emph{Given a graph }$F,$\emph{ what is the maximum number
of edges of a graph of order }$n,$\emph{ with no subgraph isomorphic to
}$F?\medskip$

Such problems are fairly well understood nowadays; see, e.g., \cite{Bol78} for
comprehensive discussion and \cite{Nik11} for some newer results. During the
past two decades, some subtler versions of Problem A have been investigated,
namely for $\lambda\left(  A\left(  G\right)  \right)  $ and $\lambda\left(
Q\left(  G\right)  \right)  $. In these problems, the central questions are
the following ones:\medskip

\textbf{Problem B }\emph{Given a graph }$F,$\emph{ what is the maximum
}$\lambda\left(  A\left(  G\right)  \right)  $\emph{ of a graph }$G$ \emph{of
order }$n,$\emph{ with no subgraph isomorphic to }$F?\medskip$

\textbf{Problem C }\emph{Given a graph }$F,$\emph{ what is the maximum
}$\lambda\left(  Q\left(  G\right)  \right)  $\emph{ of a graph }$G$ \emph{of
order }$n,$\emph{ with no subgraph isomorphic to }$F?\medskip$

Many instances of Problem B have been solved, see, e.g., the second part of
the survey paper \cite{Nik11}. There is also considerable progress with
Problem C: see, e.g., the papers \cite{AbNi12}, \cite{AbNi13}, \cite{FNP13},
\cite{FNP16}, \cite{HJZ13}, \cite{Nik14}, \cite{NiYu14}, \cite{NiYu15}, and
\cite{Yua14}.

Now, having the family $A_{\alpha}\left(  G\right)  $, we can merge Problems B
and C into one, namely:$\medskip$

\textbf{Problem D }\emph{Given a graph }$F,$\emph{ what is the maximum
}$\lambda\left(  A_{\alpha}\left(  G\right)  \right)  $\emph{ of a graph }$G$
\emph{of order }$n,$\emph{ with no subgraph isomorphic to }$F?\medskip$

In this survey we shall solve Problem D when $F$ is a complete graph. Several,
other cases seem particularly interesting:

\begin{problem}
Solve problem D if $F$ is a path or a cycle of given order.
\end{problem}

\subsection{Chromatic number and $\lambda\left(  A_{\alpha}\left(  G\right)
\right)  $}

A graph is called $r$\emph{-}$c\emph{hromatic}$ (or $\emph{r}$\emph{-partite})
if its vertices can be partitioned into $r$ edgeless sets.\emph{ }An
interesting topic in spectral graph theory is to find eigenvalues bounds on
the chromatic number of graphs. In particular, here we are interested in the
maximum $\lambda\left(  A_{\alpha}\left(  G\right)  \right)  $ if $G$ is an
$r$-partite graph of order $n$.

Let us write $T_{r}\left(  n\right)  $ for the $r$-partite Tur\'{a}n graph of
order $n$ and\ recall that $T_{r}\left(  n\right)  $ is a complete $r$-partite
graph of order $n$, whose partition sets are of size $\left\lfloor
n/r\right\rfloor $ or $\left\lceil n/r\right\rceil $. Note that for $r=2$ we
have $T_{2}\left(  n\right)  =K_{\left\lfloor n/2\right\rfloor ,\left\lceil
n/2\right\rceil }.$ It is known that $T_{r}\left(  n\right)  $ has the maximum
number of edges among all $r$-partite graphs of order $n.$ The corresponding
problem for $\lambda\left(  A_{\alpha}\left(  G\right)  \right)  $ is not so
straightforward, so for reader's sake we shall consider the case $r=2$ first.

\begin{theorem}
Let $G$ be a bipartite graph of order $n$.

(i) If $\alpha<1/2,$ then
\[
\lambda\left(  A_{\alpha}\left(  G\right)  \right)  <\lambda\left(  A_{\alpha
}\left(  T_{2}\left(  n\right)  \right)  \right)  ,
\]
unless $G=T_{2}\left(  n\right)  .$

(ii) If $\alpha>1/2,$ then
\[
\lambda\left(  A_{\alpha}\left(  G\right)  \right)  <\lambda\left(  A_{\alpha
}\left(  K_{1,n-1}\right)  \right)  ,
\]
unless $G=K_{1,n-1}.$

(iii) If $\alpha=1/2,$ then
\[
\lambda\left(  A_{\alpha}\left(  G\right)  \right)  \leq n/2,
\]
with equality if and only if $G$ is a complete bipartite graph.
\end{theorem}

\begin{proof}
Suppose that $G$ is a bipartite graph of order $n$ with maximum $\lambda
\left(  A_{\alpha}\left(  G\right)  \right)  $ among all bipartite graphs of
order $n.$ Proposition \ref{proPF} implies that $G$ is a complete bipartite
graph. Suppose that the partition sets $V_{1}$ and $V_{2}$ of $G$ are of size
$n_{1}$ and $n_{2},$ where $n_{1}+n_{2}=n.$ Set $\lambda=\lambda\left(
A_{\alpha}\left(  G\right)  \right)  $ and let $\left(  x_{1},\ldots
,x_{n}\right)  $ be a positive eigenvector to $\lambda.$ Proposition
\ref{eqth} implies that entries corresponding to vertices in the same
partition set have the same value, say $z_{i}$ for $V_{i},$ $i=1,2.$ So the
equations (\ref{eeq}) give
\begin{align*}
\lambda z_{1}  &  =\alpha n_{2}z_{1}+\left(  1-\alpha\right)  n_{2}z_{2},\\
\lambda z_{2}  &  =\alpha n_{1}x_{2}+\left(  1-\alpha\right)  n_{1}z_{1}.
\end{align*}
Excluding $z_{1}$ and $z_{2},$ we find that
\[
\left(  \lambda-\alpha n_{2}\right)  \left(  \lambda-\alpha n_{1}\right)
=\left(  1-\alpha\right)  ^{2}n_{1}n_{2}%
\]
and therefore,%
\[
\lambda=\frac{\alpha n+\sqrt{\alpha^{2}n^{2}+4n_{1}n_{2}\left(  1-2\alpha
\right)  }}{2}.
\]
Clearly if $\alpha<1/2,$ then $\lambda$ is maximum whenever $n_{1}n_{2}$ is
maximum; hence $G=T_{2}\left(  n\right)  $. Likewise if $\alpha>1/2,$ then
$\lambda$ is maximum whenever $n_{1}n_{2}$ is minimum, and so $G=K_{1,n-1}.$
Finally if $\alpha=1/2,$ then $\lambda=n/2$ for every complete bipartite graph.
\end{proof}

For general $r$ the statement reads as:

\begin{theorem}
\label{tcol}Let $r\geq2$ and $G$ be an $r$-chromatic graph of order $n$.

(i) If $\alpha<1-1/r,$ then
\[
\lambda\left(  A_{\alpha}\left(  G\right)  \right)  <\lambda\left(  A_{\alpha
}\left(  T_{r}\left(  n\right)  \right)  \right)  ,
\]
unless $G=T_{r}\left(  n\right)  .$

(ii) If $\alpha>1-1/r,$ then
\[
\lambda\left(  A_{\alpha}\left(  G\right)  \right)  <\lambda\left(  A_{\alpha
}\left(  S_{n,r-1}\right)  \right)  ,
\]
unless $G=S_{n,r-1}.$

(iii) If $\alpha=1-1/r,$ then
\[
\lambda\left(  A_{\alpha}\left(  G\right)  \right)  \leq\left(  1-1/r\right)
n,
\]
with equality if and only if $G$ is a complete $r$-partite graph.
\end{theorem}

\begin{proof}
Suppose that $G$ is an $r$-partite graph of order $n$ with maximum
$\lambda\left(  A_{\alpha}\left(  G\right)  \right)  $ among all $r$-partite
graphs of order $n.$ Proposition \ref{proPF} implies that $G$ is a complete
$r$-partite graph. Suppose that $V_{1},\ldots,V_{r}$ are the partition sets of
$G,$ with sizes $n_{1},\ldots,n_{r};$ obviously $n_{1}+\cdots+n_{r}=n.$ Set
$\lambda:=\lambda\left(  A_{\alpha}\left(  G\right)  \right)  $ and let
$\left(  x_{1},\ldots,x_{n}\right)  $ be a positive eigenvector to $\lambda$.
Proposition \ref{eqth} implies that the entries of $\mathbf{x}$ corresponding
to vertices in the same partition set have the same value, say $z_{i}$ for
$V_{i},$ $i=1,\ldots,r.$ Hence, equations (\ref{eeq}) reduce to $r$ equations%
\begin{equation}
\lambda z_{k}=\alpha\left(  n-n_{k}\right)  z_{k}+\left(  1-\alpha\right)
\sum_{i\in\left[  r\right]  \backslash\left\{  k\right\}  }n_{i}z_{i},\text{
\ \ }1\leq k\leq r. \label{req}%
\end{equation}
If $\alpha=1-1/r,$ we see that $\lambda=\left(  1-1/r\right)  n$ always is an
eigenvalue with an eigenvector defined by $z_{i}=1/\left(  rn_{i}\right)  ,$
$i=1,\ldots,r.$ This proves \emph{(iii). }

Further, letting $S=n_{1}z_{1}+\cdots+n_{r}z_{r},$ equations (\ref{req}) imply
that
\[
\left(  \lambda-\alpha\left(  n-n_{k}\right)  +\left(  1-\alpha\right)
n_{k}\right)  n_{k}z_{k}=\left(  1-\alpha\right)  n_{k}S,\text{ \ \ }1\leq
k\leq r.
\]
After some algebra, we see that $\lambda$ satisfies the equation
\begin{equation}
\sum_{k\in\left[  r\right]  }\frac{n_{k}}{\lambda-\alpha n+n_{k}}=\frac
{1}{1-\alpha}. \label{che}%
\end{equation}

If $\alpha<1-1/r,$ then $1/\left(  1-\alpha\right)  <r.$ Hence some of the
summands in the right side of (\ref{che}) is less than $1$ and so
$\lambda-\alpha n>0.$ Letting
\[
f\left(  z\right)  :=\frac{z}{\lambda-\alpha n+z}=1-\frac{\lambda-\alpha
n}{\lambda-\alpha n+z}\text{ },
\]
it is easy to see that
\[
f^{\prime\prime}\left(  z\right)  =\frac{-2\left(  \lambda-\alpha n\right)
}{\left(  \lambda-\alpha n+z\right)  ^{3}}<0
\]
for $z>0$; thus $f\left(  z\right)  $ is concave for $z>0.$

Let $\lambda_{T}:=\lambda\left(  A_{\alpha}\left(  T_{r}\left(  n\right)
\right)  \right)  $ and let $t_{1},\ldots,t_{r}$ be the sizes of the partition
sets of $T_{r}\left(  n\right)  ,$ that is to say, $t_{i}=\left\lfloor
n/r\right\rfloor $ or $t_{i}=\left\lceil n/r\right\rceil $ and $t_{1}%
+\cdots+t_{r}=n.$ In view of (\ref{che}) we have
\[
\sum_{k\in\left[  r\right]  }\frac{t_{k}}{\lambda_{T}-\alpha n+t_{k}}=\frac
{1}{1-\alpha}.
\]
Now the concavity of $f\left(  z\right)  $ implies that
\[
\sum_{k\in\left[  r\right]  }\frac{t_{k}}{\lambda_{T}-\alpha n+t_{k}}=\frac
{1}{1-\alpha}=\sum_{k\in\left[  r\right]  }\frac{n_{k}}{\lambda-\alpha
n+n_{k}}\leq\sum_{k\in\left[  r\right]  }\frac{t_{k}}{\lambda-\alpha n+t_{k}%
}.
\]
and so $\lambda_{T}\geq\lambda,$ with equality if and only if $n_{i}=$
$\left\lfloor n/r\right\rfloor $ or $n_{i}=\left\lceil n/r\right\rceil $ for
all $i\in\left[  r\right]  .$ This proves \emph{(i)}.

The proof of \emph{(ii)} goes along the same lines. If $\alpha>1-1/r,$ then
$1/\left(  1-\alpha\right)  >r.$ Hence some of the summands in the right side
of (\ref{che}) is greater than $1$ and so $\lambda-\alpha n<0.$ Letting
\[
f\left(  z\right)  :=\frac{z}{\lambda-\alpha n+z},
\]
it is easy to see that $f^{\prime\prime}\left(  z\right)  >0$ for $z>0;$ thus
$f\left(  z\right)  $ is convex for $z>0.$

Let $\lambda_{S}:=\lambda\left(  A_{\alpha}\left(  S_{n,r-1}\right)  \right)
$ and let $s_{1},\ldots,s_{r}$ be the sizes of the partition sets of
$S_{n,r-1},$ that is to say, $s_{1}=\cdots=s_{r-1}=1$ and $s_{r}=n-r+1.$ In
view of (\ref{che}), we have
\[
\sum_{k\in\left[  r\right]  }\frac{s_{k}}{\lambda_{S}-\alpha n+s_{k}}=\frac
{1}{1-\alpha}.
\]
Now the convexity of $f\left(  z\right)  $ implies that
\[
\sum_{k\in\left[  r\right]  }\frac{s_{k}}{\lambda_{S}-\alpha n+s_{k}}=\frac
{1}{1-\alpha}=\sum_{k\in\left[  r\right]  }\frac{n_{k}}{\lambda-\alpha
n+n_{k}}\leq\sum_{k\in\left[  r\right]  }\frac{s_{k}}{\lambda-\alpha n+s_{k}%
}.
\]
and so $\lambda_{S}\geq\lambda,$ with equality if and only if one partition
set of $G$ is of size $n-r+1,$ and the rest are of size 1, that is to say
$G=S_{n,r-1}$. The proof of Theorem \ref{tcol} is completed.
\end{proof}

\subsection{Clique number and $\lambda\left(  A_{\alpha}\left(  G\right)
\right)  $}

A graph is called $K_{r}$\emph{-free} if it does not contain a complete graph
on $r$ vertices. It is known (see, e.g., \cite{Nik11} and \cite{HJZ13}) that
if $G$ is a $K_{r+1}$-free graph of \ order $n,$ then
\begin{align*}
\lambda\left(  A\left(  G\right)  \right)   &  \leq\lambda\left(  A\left(
T_{r}\left(  n\right)  \right)  \right)  ,\\
\lambda\left(  Q\left(  G\right)  \right)   &  \leq\lambda\left(  Q\left(
T_{r}\left(  n\right)  \right)  \right)  .
\end{align*}
The generalization of these results to $\lambda\left(  A_{\alpha}\left(
G\right)  \right)  $ turned out to be quite unexpected, and is summarized in
the following encompassing theorem:

\begin{theorem}
\label{tTur}Let $r\geq2$ and $G$ be an $K_{r+1}$-free graph of order $n$.

(i) If $0\leq\alpha<1-1/r,$ then
\[
\lambda\left(  A_{\alpha}\left(  G\right)  \right)  <\lambda\left(  A_{\alpha
}\left(  T_{r}\left(  n\right)  \right)  \right)  ,
\]
unless $G=T_{r}\left(  n\right)  .$

(ii) If $1>\alpha>1-1/r,$ then
\[
\lambda\left(  A_{\alpha}\left(  G\right)  \right)  <\lambda\left(  A_{\alpha
}\left(  S_{n,r-1}\right)  \right)  ,
\]
unless $G=S_{n,r-1}.$

(iii) If $\alpha=1-1/r,$ then
\[
\lambda\left(  A_{\alpha}\left(  G\right)  \right)  \leq\left(  1-1/r\right)
n,
\]
with equality if and only if $G$ is a complete $r$-partite graph.
\end{theorem}

We shall show that Theorem \ref{tTur} can be reduced to Theorem \ref{tcol} via
a technical lemma.

\begin{lemma}
\label{rlem}Let $\alpha\in\left[  0,1\right)  $ and $n\geq r\geq2.$ If $G$ is
a graph with maximum $\lambda\left(  A_{\alpha}\left(  G\right)  \right)  $
among all $K_{r+1}$-free graphs of order $n,$\ then $G$ is complete $r$-partite.
\end{lemma}

For the proof of the lemma, we introduce some notation: Let $\alpha\in\left[
0,1\right)  .$ Given a graph $G$ of order $n$ and a vector $\mathbf{x}%
:=\left(  x_{1},\ldots,x_{n}\right)  ,$ set
\[
S_{G}\left(  \mathbf{x}\right)  :=\left\langle A_{\alpha}\left(  G\right)
\mathbf{x},\mathbf{x}\right\rangle ,
\]
and for any $v\in V\left(  G\right)  $, set
\[
S_{G}(v,\mathbf{x}):=\alpha d_{G}\left(  u\right)  +\left(  1-\alpha\right)
\sum_{\left\{  v,i\right\}  \in E\left(  G\right)  }x_{i}.
\]

\begin{proof}
[\textbf{Proof of Lemma \ref{rlem}}]Let $G$ be a graph with maximum
$\lambda\left(  A_{\alpha}\left(  G\right)  \right)  $ among all $K_{r+1}%
$-free graphs of order $n$. For short, let $\lambda:=\lambda\left(  A_{\alpha
}\left(  G\right)  \right)  $. Clearly $G$ is connected, so there is a
positive unit eigenvector $\mathbf{x}:=\left(  x_{1},\ldots,x_{n}\right)  $ to
$\lambda\left(  A_{\alpha}\left(  G\right)  \right)  $, and therefore,%
\[
\lambda=S_{G}\left(  \mathbf{x}\right)  =\sum\limits_{v\in V(G)}x_{v}%
S_{G}(v,\mathbf{x}).
\]
Note that the eigenequation (\ref{eeq}) for any vertex $v\in V\left(
G\right)  $ can be written as
\begin{equation}
\lambda x_{u}=S_{G}(v,\mathbf{x}). \label{eq2}%
\end{equation}
To prove the lemma we need two claims.

\textbf{Claim A} \emph{There exists a coclique }$W\subset G$\emph{ such that }%
\[
G=W\vee G^{\prime},
\]
\emph{where }$G^{\prime}=G-V\left(  G_{1}\right)  $\emph{.}

\emph{Proof }Select a vertex $u$ with
\[
S_{G}(u,\mathbf{x}):=\max\left\{  {S_{G}(v,\mathbf{x}):v\in V(G)}\right\}  ,
\]
and set $U:=\Gamma_{G}(u)$ and $W:=G-U$. Remove all edges within $W$ and join
each vertex in $U$ to each vertex in $W.$ Write $H$ for the resulting graph,
which is obviously of order $n$ and is $K_{r+1}$-free. We shall show that
$S_{H}\left(  v,\mathbf{x}\right)  \geq S_{G}\left(  v,\mathbf{x}\right)  $
for each $v\in V\left(  G\right)  .$ This is obvious if $v\in U$, since then
$\Gamma_{G}\left(  v\right)  \subset\Gamma_{H}\left(  v\right)  ,$ and so
$S_{H}(v,\mathbf{x})\geq S_{G}(v,\mathbf{x}).$ Now, let $v\in V\left(
W\right)  .$ Note that
\[
S_{H}(v,\mathbf{x})=\alpha d_{G}\left(  u\right)  x_{v}+\left(  1-\alpha
\right)  \sum_{\left\{  u,i\right\}  \in E\left(  G\right)  }x_{i}=\alpha
d_{G}\left(  u\right)  x_{v}+S_{G}(u,\mathbf{x})-\alpha d_{G}\left(  u\right)
x_{u}.
\]
Hence,
\[
S_{H}(v,\mathbf{x})-S_{G}(v,\mathbf{x})=S_{G}(u,\mathbf{x})-S_{G}%
(v,\mathbf{x})-\alpha d_{G}\left(  u\right)  \left(  x_{u}-x_{v}\right)  .
\]
Now, equation (\ref{eq2}) implies that $S_{G}(u,\mathbf{x})=\lambda x_{u}$ and
$S_{G}(v,\mathbf{x})=\lambda x_{v}.$ Hence
\[
S_{H}(v,\mathbf{x})-S_{G}(v,\mathbf{x})=\lambda\left(  x_{u}-x_{v}\right)
-\alpha d_{G}\left(  u\right)  \left(  x_{u}-x_{v}\right)  =\left(
\lambda-\alpha d_{G}\left(  u\right)  \right)  \left(  x_{u}-x_{v}\right)  .
\]
But Corollary \ref{corlo} implies that $\lambda-\alpha d_{G}\left(  u\right)
>0,$ and equation (\ref{eq2}) implies that $x_{u}\geq x_{v}.$ Hence
$S_{H}(v,\mathbf{x})\geq S_{G}(v,\mathbf{x})$ for any $v\in V\left(  G\right)
,$ and so
\[
\lambda\left(  A_{\alpha}\left(  H\right)  \right)  \geq S_{H}\left(
\mathbf{x}\right)  \geq S_{H}\left(  \mathbf{x}\right)  =\lambda\geq
\lambda\left(  A_{\alpha}\left(  H\right)  \right)  .
\]
Therefore, $\lambda\left(  A_{\alpha}\left(  H\right)  \right)  =\lambda,$
implying, in particular, that $S_{H}(v,\mathbf{x})=S_{G}(v,\mathbf{x})$ for
each $v\in U;$ thus each $v\in U$ is joined in $G$ to each $w\in W$, and so
$G=H=W\vee G\left[  U\right]  ,$ completing the proof of Claim A.

To finish the proof of the lemma we need another technical assertion:\medskip

\textbf{Claim} \textbf{B }\emph{Let }$1\leq k<r.$ \emph{If }$F$\emph{ is an
induced subgraph of }$G$ \emph{and }$W_{1},\ldots,W_{k}$\emph{ are disjoint
cocliques of }$G$ \emph{such that} \emph{ }%
\[
G=W_{1}\vee\cdots\vee W_{k}\vee F
\]
\emph{then there is a coclique }$W_{k+1}\subset F$\emph{ such that }%
\[
G=W_{1}\vee\cdots\vee W_{k+1}\vee F^{\prime},
\]
\emph{where }$F^{\prime}=F-V\left(  W_{k+1}\right)  $\emph{.}

\emph{Proof }Select a vertex $v\in{V(F)}$ with
\[
S_{G}(u,\mathbf{x})=\max\left\{  {S_{G}(v,\mathbf{x}):v\in V(F)}\right\}  ,
\]
and set $U:=\Gamma_{F}(u)$ and $W:=F-U$. Remove all edges within $W$ and join
each vertex in $U$ to each vertex in $W.$ Write $H$ for the resulting graph,
which is obviously of order $n$ and is $K_{r+1}$-free. We shall show that
$S_{H}\left(  v,\mathbf{x}\right)  \geq S_{G}\left(  v,\mathbf{x}\right)  $
for each $v\in V\left(  G\right)  .$ This is obvious if $v\in V\backslash
V\left(  W\right)  ,$ since then either $\Gamma_{G}\left(  v\right)
=\Gamma_{H}\left(  v\right)  $ or $\Gamma_{G}\left(  v\right)  \subset
\Gamma_{H}\left(  v\right)  ,$ and so $S_{H}(v,\mathbf{x})\geq S_{rG}%
(v,\mathbf{x}).$ Now, let $v\in V\left(  W\right)  .$ Exactly as in the proof
of Claim A we see that
\[
S_{H}(v,\mathbf{x})-S_{G}(v,\mathbf{x})=\left(  \lambda-\alpha d_{G}\left(
u\right)  \right)  \left(  x_{u}-x_{v}\right)  .
\]
Hence, $S_{H}(v,\mathbf{x})\geq S_{G}(v,\mathbf{x})$ and%
\[
\lambda\left(  A_{\alpha}\left(  H\right)  \right)  \geq S_{H}\left(
\mathbf{x}\right)  \geq S_{H}\left(  \mathbf{x}\right)  =\lambda\geq
\lambda\left(  A_{\alpha}\left(  H\right)  \right)  .
\]
Therefore, $\lambda\left(  A_{\alpha}\left(  H\right)  \right)  =\lambda,$
implying, in particular, that $S_{H}(v,\mathbf{x})=S_{G}(v,\mathbf{x})$ for
each $v\in U;$ thus each $v\in U$ is joined in $H_{k}$ to each $w\in W$, and
so $F=W\vee G\left[  U\right]  .$ Letting $W_{k+1}:=W,$ the proof of Claim B
is completed.

To complete the proof of the lemma, we first apply Claim A and then repeatedly
apply Claim B until $k=r-2.$ In this way we find that%
\[
G=W_{1}\vee\cdots\vee W_{r-1}\vee F,
\]
where $W_{1}\vee\cdots\vee W_{r-1}$ are cocliques of $G$ and $F$ is an induced
subgraph of $G.$ Because $G$ is $K_{r+1}$-free, $F$ must be a coclique too and
so, $G$ is a complete $r$-partite graph.
\end{proof}

\section{\label{smis}Miscellaneous}

In this section we briefly touch a few rather different topics, some of which
deserve a much more thorough investigation. 

\subsection{The smallest eigenvalue $\lambda_{\min}\left(  A_{\alpha}\left(
G\right)  \right)  $}

The smallest eigenvalue of the adjacency matrix, which is second in importance
after the spectral radius, has numerous relations with the structure of the
graph. To a great extent this is also true for $\lambda_{\min}\left(  Q\left(
G\right)  \right)  $; see, e.g., \cite{DR94}, \cite{LOAN11}, and \cite{LNO16}.
In particular, the smallest eigenvalues of $A\left(  G\right)  $ and $Q\left(
G\right)  $ have close relations to bipartite subgraphs of $G$. A simple
relation of this type can be obtained also for $\lambda_{\min}\left(
A_{\alpha}\left(  G\right)  \right)  .$

Let $G$ be a graph of order $n$ with $m$ edges. Let $V\left(  G\right)
=V_{1}\cup V_{2}$ be a bipartition and let the $n$-vector $\mathbf{x}:=\left(
x_{1},\ldots,x_{n}\right)  $ be $-1$ on $V_{1}$ and $1$ on $V_{2}.$ We see
that
\[
\left\langle A_{\alpha}\left(  G\right)  \mathbf{x},\mathbf{x}\right\rangle
=2\alpha m-2\left(  1-\alpha\right)  e\left(  V_{1},V_{2}\right)
\]
Hence, scaling $\left(  x_{1},\ldots,x_{n}\right)  $ to unit length, we get:

\begin{proposition}
If $G$ is a graph of order $n$ with $m$ edges, then
\[
\lambda_{\min}\left(  A_{\alpha}\left(  G\right)  \right)  \leq2\alpha\frac
{m}{n}-\frac{2\left(  1-\alpha\right)  }{n}\text{ \textrm{maxcut}}(G).
\]

\end{proposition}

It is interesting to determine the minimum value of $\lambda_{\min}\left(
A_{\alpha}\left(  G\right)  \right)  $ if $G$ is a graph of order $n.$ For
$\alpha\geq1/2$ this is easy. Indeed, if $\alpha\geq1/2,$ the matrix
$A_{\alpha}\left(  G\right)  $ is positive semidefinite, and so $\lambda
_{\min}\left(  A_{\alpha}\left(  G\right)  \right)  \geq0.$ On the other hand,
if $G$ has an isolated vertex, then $\lambda_{\min}\left(  A_{\alpha}\left(
G\right)  \right)  =0,$ so if $\alpha\in\left[  1/2,1\right]  ,$ then
\[
\min\left\{  \lambda_{\min}\left(  A_{\alpha}\left(  G\right)  \right)
:v\left(  G\right)  =n\right\}  =0.
\]
By contrast,%
\[
\min\left\{  \lambda_{\min}\left(  A\left(  G\right)  \right)  :v\left(
G\right)  =n\right\}  =-\sqrt{\left\lfloor n/2\right\rfloor \left\lceil
n/2\right\rceil }\text{;}%
\]
hence it is worth to raise the following problem:

\begin{problem}
For any $\alpha\in\left(  0,1/2\right)  $ determine%
\[
\min\left\{  \lambda_{\min}\left(  A_{\alpha}\left(  G\right)  \right)
:v\left(  G\right)  =n\right\}  .
\]

\end{problem}

\subsection{The second largest eigenvalue $\lambda_{2}\left(  A_{\alpha
}\left(  G\right)  \right)  $}

In this subsection we discuss how large $\lambda_{2}\left(  A_{\alpha}\left(
G\right)  \right)  $ can be if $G$ is a graph of order $n.$  

\begin{proposition}
Let $G$ be a graph of order $n$ with $A_{\alpha}\left(  G\right)  =A_{\alpha
}.$

(a) If $1/2\leq\alpha\leq1,$ then%
\[
\lambda_{2}\left(  A_{\alpha}\right)  \leq\alpha n-1.
\]
If $\alpha>1/2,$ equality is attained if and only if $G=K_{n}.$

(b) If $0\leq\alpha<1/2,$ then%
\[
\lambda_{2}\left(  A_{\alpha}\right)  \leq\frac{n}{2}-1.
\]
If $n$ is even equality holds for the graph $G=2K_{n/2}.$ 
\end{proposition}

Note that we have not determined precisely how large $\lambda_{2}\left(
A_{\alpha}\left(  G\right)  \right)  $ can be if $G$ is a graph of odd order
$n.$ Taking $G=K_{\left[  n/2\right]  }\cup K_{\left\lceil n/2\right\rceil },$
we see that
\[
\lambda_{2}\left(  A_{\alpha}\left(  G\right)  \right)  =\frac{n-1}{2}-1,
\]
but this still leaves a margin of $1/2$ to close.

\subsection{Eigenvalues of $A_{\alpha}\left(  G\right)  $ and the diameter of
$G$}

The following theorem can be proved using the generic idea of \cite{Cve08}.

\begin{proposition}
Let $a\in\left[  0,1\right)  $, let $G$ be a graph with $A_{\alpha}\left(
G\right)  =A_{\alpha}$, and let $u$ and $v$ be two vertices of $G$ at distance
$k\geq1.$ Let $l\in\left[  k\right]  $ and set $B:=A_{\alpha}^{l}.$

(a) If $l=k$, then $b_{u,v}>0$;

(b) If $l<k,$ then $b_{u,v}=0.$
\end{proposition}

\begin{proof}
Set $A=A\left(  G\right)  .$ If $X$ and $Y$ are matrices of the same size,
write $X\succ Y,$ if $x_{i,j}\geq y_{i,j}$ for all admissible $i,j.$

\emph{Proof of (a)} Note that $A_{\alpha}\succ\left(  1-\alpha\right)  A,$ and
so $A_{\alpha}^{k}\succ\left(  1-\alpha\right)  ^{k}A^{k}.$ However, the
$(u,v)$ entry of $A^{k}$ is positive, since there is a path of length $k$
between $u$ and $v.$ Hence, $b_{u,v}>0,$ proving \emph{(a)}.

\emph{Proof of (b) }Now suppose that $l<k,$ and note that $A+nI\succ
A_{\alpha}.$ Hence, $\left(  A+nI\right)  ^{l}\succ A_{\alpha}^{l}.$ Since
\[
\left(  A+nI\right)  ^{l}=A^{l}+a_{l-1}A^{l-1}+\cdots+a_{0}I^{l}%
\]
for some real $a_{0},\ldots,a_{l-1},$ we see that the $(u,v)$ entry of
\ $\left(  A+nI\right)  ^{l}$ is zero, because there is no path shorter than
$k$ between $u$ and $v,$ and so the $(u,v)$ entry of each of the matrices
$A^{l},\ldots,A,I$ is zero. Hence, $b_{u,v}=0.$
\end{proof}

\begin{corollary}
If $G$ is a connected graph of diameter $D$, then $A_{\alpha}\left(  G\right)
$ has at least $D+1$ distinct eigenvalues.
\end{corollary}

\subsection{Eigenvalues of $A_{\alpha}\left(  G\right)  $ and traces}

In this subsection we give two explicit expressions for the sums and the sum
of squares of the eigenvalues of $A_{\alpha}\left(  G\right)  .$ 

\begin{proposition}
If $G$ is a graph of order $n$ and has $m$ edges, then
\[
\sum_{i=1}^{n}\lambda_{i}\left(  A_{\alpha}\left(  G\right)  \right)
=\mathrm{tr}\text{ }A_{\alpha}\left(  G\right)  =\alpha\sum_{i\in V\left(
G\right)  }d_{G}\left(  u\right)  =2\alpha m.
\]

\end{proposition}

Here is a similar formula for the sum of the squares of the $A_{\alpha}$-eigenvalues.

\begin{proposition}
If $G$ is a graph of order $n$ and has $m$ edges, then$.$%
\[
\sum_{i=1}^{n}\lambda_{i}^{2}\left(  A_{\alpha}\left(  G\right)  \right)
=\mathrm{tr}\text{ }A_{\alpha}^{2}\left(  G\right)  =2\left(  1-\alpha\right)
^{2}m+\alpha^{2}\sum_{i\in V}d_{G}^{2}\left(  u\right)  .
\]

\end{proposition}

\begin{proof}
Let $A_{\alpha}:=A_{\alpha}\left(  G\right)  $, $A:=A\left(  G\right)  $, and
$D:=D\left(  G\right)  $. Calculating the square $A_{\alpha}^{2}$ and taking
its trace, we find that%
\begin{align*}
\mathrm{tr}\text{ }A_{\alpha}^{2} &  =\mathrm{tr}\text{ }(\alpha^{2}%
D^{2}+\left(  1-\alpha\right)  ^{2}A^{2}+\alpha\left(  1-\alpha\right)
DA+\alpha\left(  1-\alpha\right)  AD)\\
&  =\alpha^{2}\mathrm{tr}\text{ }D^{2}+\left(  1-\alpha\right)  ^{2}%
\mathrm{tr}\text{ }A^{2}+\alpha\left(  1-\alpha\right)  \mathrm{tr}\text{
}DA+\alpha\left(  1-\alpha\right)  \mathrm{tr}\text{ }AD\\
&  =2\left(  1-\alpha\right)  ^{2}m+\alpha^{2}\sum_{i\in V}d_{G}^{2}\left(
u\right)  ,
\end{align*}
completing the proof.
\end{proof}

\section{\label{ssg}The $A_{\alpha}$-spectra of some graphs}

Equalities (\ref{reg}) and the fact the eigenvalues of $A\left(  K_{n}\right)
$ are $\left\{  n-1,-1,\ldots,-1\right\}  $ give the spectrum of $A_{\alpha
}\left(  K_{n}\right)  $ as follows:

\begin{proposition}
The eigenvalues of $A_{\alpha}\left(  K_{n}\right)  $ are
\[
\lambda_{1}\left(  A_{\alpha}\left(  K_{n}\right)  \right)  =n-1\text{ \ \ and
\ \ }\lambda_{k}\left(  A_{\alpha}\left(  K_{n}\right)  \right)  =\alpha
n-1\text{ for }2\leq k\leq n.
\]

\end{proposition}

Next, we present the $A_{\alpha}$-spectrum of the complete bipartite graph
$K_{a,b}$, but we omit the proof.

\begin{proposition}
\label{procb}Let $a\geq b\geq1.$ If $\alpha\in\left[  0,1\right]  ,$ the
eigenvalues of $A_{\alpha}\left(  K_{a,b}\right)  $ are
\begin{align*}
\lambda\left(  A_{\alpha}\left(  K_{a,b}\right)  \right)   &  =\frac{1}%
{2}\left(  \alpha\left(  a+b\right)  +\sqrt{\alpha^{2}\left(  a+b\right)
^{2}+4ab\left(  1-2\alpha\right)  }\right)  ,\\
\lambda_{\min}\left(  A_{\alpha}\left(  K_{a,b}\right)  \right)   &  =\frac
{1}{2}\left(  \alpha\left(  a+b\right)  -\sqrt{\alpha^{2}\left(  a+b\right)
^{2}+4ab\left(  1-2\alpha\right)  }\right)  ,\\
\lambda_{k}\left(  A_{\alpha}\left(  K_{a,b}\right)  \right)   &  =\alpha
a\text{ \ for }1<k\leq b,\\
\lambda_{k}\left(  A_{\alpha}\left(  K_{a,b}\right)  \right)   &  =\alpha
b\text{ \ for }b<k<a+b.
\end{align*}

\end{proposition}

In particular, the $A_{\alpha}$-spectrum of the star $K_{1,n-1}$ is as follows:

\begin{proposition}
\label{prost}The eigenvalues of $A_{\alpha}\left(  K_{1,n-1}\right)  $ are
\begin{align*}
\lambda\left(  A_{\alpha}\left(  K_{1,n-1}\right)  \right)   &  =\frac{1}%
{2}\left(  \alpha n+\sqrt{\alpha^{2}n^{2}+4\left(  n-1\right)  \left(
1-2\alpha\right)  }\right) \\
\text{\ }\lambda_{\min}\left(  A_{\alpha}\left(  K_{1,n-1}\right)  \right)
&  =\frac{1}{2}\left(  \alpha n-\sqrt{\alpha^{2}n^{2}+4\left(  n-1\right)
\left(  1-2\alpha\right)  }\right) \\
\lambda_{k}\left(  A_{\alpha}\left(  K_{1,n-1}\right)  \right)   &
=\alpha\text{ \ \ for }1<k<n.
\end{align*}

\end{proposition}

\section{Concluding remarks}

This survey covers just a small portion of the hundreds of results about
$A\left(  G\right)  $ and $Q\left(  G\right)  $ that could be extended to
$A_{\alpha}\left(  G\right)  .$ This is a challenging endeavor. If nothing
else, Theorems \ref{tcol} and \ref{tTur} show that it is worth studying
$A_{\alpha}\left(  G\right)  $, for it is difficult to discover them in a
different context.

\end{document}